\documentclass{amsart}
\usepackage{graphicx}
\usepackage{amsfonts}
\newtheorem{tm}{Theorem}

\newtheorem{convention}{Convention}
\newtheorem{rem}{Remark}
\newtheorem{rems}{Remarks}
\newtheorem{lm}{Lemma}
\newtheorem{ex}{Example}

\newtheorem{nota}{Notation}

\begin{document}

\title{On realizability of sign patterns by real polynomials}
\author{Vladimir Petrov Kostov}
\address{Universit\'e C\^ote d’Azur, CNRS, LJAD, France} 
\email{vladimir.kostov@unice.fr}
\begin{abstract}
The classical Descartes' rule of signs limits the 
number of positive roots of a real polynomial in one variable by the 
number of sign changes in the sequence of its coefficients.  
One can ask 
the question which pairs of nonnegative integers $(p,n)$, chosen 
in accordance with this rule and with some other natural conditions,     
can be the pairs of numbers of positive and negative roots 
of a real polynomial with 
prescribed signs of the coefficients. The paper solves  
this problem for degree $8$ polynomials.\\ 


\end{abstract}
\maketitle 

\section{Formulation of the problem and of the results}

The classical Descartes' rule of signs states that a real polynomial in one 
variable has not more real positive roots than the number of sign changes 
in the sequence of its coefficients. 
Any sequence of $\pm$-signs 
$\bar{\sigma}:=(\sigma _0,\sigma _1,\ldots ,\sigma _d)$ is called a 
{\em sign pattern}. In the present paper 
we consider sign patterns defined by the signs of the coefficients of degree 
$d$ polynomials $P$, so in particular $\sigma _d=$sign$(P(0))$. 
For a given sign pattern its 
{\em Descartes' pair} 
$(p_{\bar{\sigma}},n_{\bar{\sigma}})$ is the number of 
sign changes and sign preservations in the sequence of coefficients. Denote 
by $(pos_P,neg_P)$ the numbers of positive and negative roots of $P$ 
counted with multiplicity. Hence the following restrictions must hold true:

\begin{equation}\label{pn}
pos_P\leq p_{\bar{\sigma}}~,~neg_P\leq n_{\bar{\sigma}}~,~
pos_P\equiv p_{\bar{\sigma}}({\rm mod}~2)~,~neg_P\equiv ,n_{\bar{\sigma}}({\rm mod}~2)~.
\end{equation}
(The inequality $neg_P\leq n_{\bar{\sigma}}$ follows from Descartes' rule 
applied to the polynomial $P(-x)$.) Pairs $(pos,neg)$ satisfying conditions 
(\ref{pn}) are called {\em admissible} for the sign pattern $\bar{\sigma}$ 
(and the latter is {\em admitting} them). 

The present paper finishes the study which was begun in 
\cite{FoKoSh} of sign patterns and their 
admissible pairs for polynomials of degree up to $8$. The present 
introduction reproduces with some small 
modifications the one of \cite{FoKoSh} and the results 
obtained in that paper, 
see Theorems~\ref{tmknown}, \ref{tmd7} and~\ref{tmd8}. The new results 
are given in Theorem~\ref{tmnew} and then 
presented in another way (suitable to be 
compared to the previously obtained ones) at the end of this section.    

Clearly conditions (\ref{pn}) are only necessary ones, i.e. for a 
given sign pattern $\bar{\sigma}$ and an admissible pair 
$(p,n)$ it is not a priori clear whether there exists a degree $d$ 
polynomial with this sign pattern and with exactly $p$ distinct positive and 
exactly $n$ distinct negative roots. If such a polynomial exists, then we 
say that the given combination of sign pattern and admissible pair 
{\em is realizable}. 

\begin{nota}
{\rm For a given sign pattern $\bar{\sigma}$ we define its corresponding 
{\em reverted} sign pattern $\bar{\sigma}^r$ as $\bar{\sigma}$ 
read from the back and by 
$\bar{\sigma}_m$ the sign pattern obtained from the given one by changing the 
signs in second, fourth, etc. position while keeping the other signs the same. 
If $\bar{\sigma}$ is 
defined by a degree $d$ polynomial 
$P(x)$, then $\bar{\sigma}^r$ is the sign pattern of $x^dP(1/x)$ and 
$\bar{\sigma}_m$ is the one of $(-1)^dP(-x)$.}
\end{nota} 

\begin{ex}
{\rm For $d=4$ the sign pattern $(+,-,-,-,+)$ is equal to 
$(+,-,-,-,+)^r$ and one has $(+,-,-,-,+)_m=(+,+,-,+,+)=(+,-,-,-,+)_m^r$. 
For $d=8$ the sign pattern $(+,+,-,+,-,-,-,-,+)$ is equal 
to $(+,+,-,+,-,-,-,-,+)^r_m$.} 
\end{ex}

\begin{rems}
{\rm (1) In what follows we assume that the leading coefficients of 
the polynomials are positive, so sign patterns (except in some places of 
the proofs) begin with $+$. 

(2) It is clear that $(\bar{\sigma}^r)^r=\bar{\sigma}$, 
$(\bar{\sigma}_m)_m=\bar{\sigma}$ and $(\bar{\sigma}^r)_m=(\bar{\sigma}_m)^r$ 
(so we write simply $\bar{\sigma}^r_m$).

(3) The sign patterns and admissible pairs $(\bar{\sigma},(p,n))$, 
$(\bar{\sigma}^r,(p,n))$, $(\bar{\sigma}_m,(n,p))$ and 
$(\bar{\sigma}^r_m,(n,p))$ are realizable or not simultaneously. Therefore 
it makes sense to consider the question of realizability of given 
sign patterns with given admissible pairs modulo the standard 
$\mathbb{Z}_2\times \mathbb{Z}_2$-action defined by 
$\bar{\sigma} \mapsto \bar{\sigma}^r$ and 
$\bar{\sigma}\mapsto \bar{\sigma}_m$.}
\end{rems}

It seems that for the first time the question of realizability of sign patterns 
with admissible pairs has been asked in \cite{AJS}. In \cite{Gr} Grabiner 
has obtained the first example of nonrealizability. Namely, he has shown that 
for $d=4$ the sign pattern $(+,-,-,-,+)$ is not realizable with the 
admissible pair $(0,2)$ (Descartes' pair of the pattern equals $(2,2)$). 
In \cite{AlFu} Albouy and Fu have given the exhaustive answer to this question 
of realizability for degrees not greater than $6$. In Theorems~\ref{tmknown}, 
\ref{tmd7} and \ref{tmd8}  
we change at some places (w.r.t. the original formulations in \cite{AlFu} 
or \cite{FoKoSh}) a sign pattern $\sigma$ to $\sigma _m$ and the 
corresponding pair $(p,n)$ to $(n,p)$ in order to have mostly pairs of the form 
$(0,n)$ in the formulations:

\begin{tm}\label{tmknown}
(1) For degree $1$, $2$ and $3$, any sign pattern is realizable with any of its 
admissible pairs.

(2) For degree $4$ the only case of nonrealizability (up to the standard 
$\mathbb{Z}_2\times \mathbb{Z}_2$-action) is the one of Grabiner's example.

(3) For degree $5$ the only such case is given by the sign pattern 
$(+,-,-,-,-,+)$ with the pair $(0,3)$.

(4) For degree $6$ the only such cases are: 
$(+,-,-,-,-,-,+)$ with $(0,2)$ or $(0,4)$; $(+,-,+,-,-,-,+)$ with $(0,2)$; 
$(+,+,-,-,-,-,+)$ with $(0,4)$.
\end{tm}

The cases $d=7$ and $d=8$ have been considered in \cite{FoKoSh}. 
The exhaustive answer to the question of realizability for $d=7$ is 
as follows:

\begin{tm}\label{tmd7}
For $d=7$ there are $1472$ cases (modulo the standard 
$\mathbb{Z}_2\times \mathbb{Z}_2$-action) of sign pattern and admissible 
pair. Of these exactly $6$ are not realizable:
$(+,+,-,-,-,-,-,+)$, $(+,+,-,-,-,-,+,+)$ and $(+,+,+,-,-,-,-,+)$ with 
$(0,5)$; 
$(+,-,-,-,-,+,-,+)$ with $(0,3)$;
$(+,-,-,-,-,-,-,+)$ with 
$(0,3)$ and $(0,5)$.
\end{tm}

For $d=8$ the partial answer from \cite{FoKoSh} 
can be summarized by the following theorem. In 
\cite{FoKoSh} this result is formulated differently, but equivalently. 
In particular, the 
authors of \cite{FoKoSh} have not noticed that the number of cases 
for which the 
answer still remained unknown can be decreased by one due to the standard 
$\mathbb{Z}_2\times \mathbb{Z}_2$-action.

\begin{tm}\label{tmd8}
(1) For $d=8$ there are $3648$ possible combinations of sign pattern 
and admissible pair (up to the standard 
$\mathbb{Z}_2\times \mathbb{Z}_2$-action). Of these exactly $13$ are known to 
be nonrealizable: 

\noindent $(+,+,-,-,-,-,-,+,+)~~,~~(+,-,-,-,-,-,-,+,+)~~,~~
(+,+,+,+,-,-,-,-,+)$ and $(+,+,+,-,-,-,-,-,+)$ with $(0,6)$; 

\noindent $(+,-,+,-,-,-,+,-,+)$ and $(+,-,+,-,+,-,-,-,+)$ with $(0,2)$; 

\noindent $(+,-,+,-,-,-,-,-,+)$ and $(+,-,-,-,+,-,-,-,+)$ 
with $(0,2)$ and $(0,4)$; 

\noindent $(+,-,-,-,-,-,-,-,+)$ with $(0,2)$, $(0,4)$ and $(0,6)$.  

(2) For exactly another $6$ cases it is not known whether they are realizable 
or not (we list the sign patterns and their reverted ones which will be needed 
later):

$$\begin{array}{lllll}
{\rm Case~1:}&&\sigma _1:=(+,+,+,-,-,-,-,+,+)&&{\rm with}~~(0,6)\\ 
&&&&\sigma _1^r=(+,+,-,-,-,-,+,+,+)\\ 
{\rm Case~2:}&&\sigma _2:=(+,+,-,+,-,-,-,+,+)&&{\rm with}~~(4,0)\\ 
&&&&\sigma _2^r=(+,+,-,-,-,+,-,+,+)\\  
{\rm Case~3:}&&\sigma _3:=(+,+,-,+,-,+,-,-,+)&&{\rm with}~~(4,0)\\ 
&&&&\sigma _3^r=(+,-,-,+,-,+,-,+,+)\\ 
{\rm Case~4:}&&\sigma _4:=(+,+,+,-,-,+,-,+,+)&&{\rm with}~~(4,0)\\ 
&&&&\sigma _4^r=(+,+,-,+,-,-,+,+,+)\\ 
{\rm Case~5:}&&\sigma _5:=(+,+,+,+,-,+,-,-,+)&&{\rm with}~~(4,0)\\ 
&&&&\sigma _5^r=(+,-,-,+,-,+,+,+,+)\\  
{\rm Case~6:}&&\sigma _6:=(+,+,-,+,-,-,-,-,+)&&{\rm with}~~(4,0)\\ 
&&&&\sigma _6^r=(+,-,-,-,-,+,-,+,+)~.
\end{array}$$
\end{tm}

The aim of the present paper is to definitely settle the case $d=8$. Namely, we 
prove the following theorem:

\begin{tm}\label{tmnew}
The $6$ cases of part (2) of Theorem~\ref{tmd8} are not realizable.
\end{tm} 

For Case~1 the proof is given in Section~\ref{seccase1}. 
Cases~2-6 are considered in Section~\ref{seccases26}. The proofs of 
Lemmas~\ref{basiclemma} and \ref{nextlemma} formulated in 
Section~\ref{seccases26} are given in the Appendix. 
In the proof of the 
theorem we sometimes use sign patterns having as components not only $+$ 
and/or $-$, but also $0$ (in the sense that the corresponding 
coefficient equals $0$), and in some cases $\pm$ meaning that 
we consider the cases with $+$ and $-$ together. 

As we see, in all cases of nonrealizability one of the components of the 
admissible pair equals $0$. The same is true for $d=9$ and $10$, 
see \cite{FoKoSh}. To finish this section we list the nonrealizable cases 
for $d=8$ by their pairs $(p,n)$; the third column contains the corresponding 
Descartes' pair. In order to have only the pairs $(0,2)$, $(0,4)$ and $(0,6)$ 
as defining the classification we change the sign patterns $\sigma _j$ of 
Cases 2-6 of Theorem~\ref{tmd8} to the corresponding patterns $(\sigma _j)_m$. 
To find easier Cases 1-6 in the table we give their numbers as indices to 
the corresponding sign patterns.

$$\begin{array}{llllll}
(0,2)&(+,-,+,-,-,-,+,-,+)&(6,2)&&(+,-,+,-,+,-,-,-,+)&(6,2)\\
 
&(+,-,+,-,-,-,-,-,+)&(4,4)&&(+,-,-,-,+,-,-,-,+)&(4,4)\\ 

&(+,-,-,-,-,-,-,-,+)&(2,6)&&&\end{array}$$ 

$$\begin{array}{llllll}
(0,4)&(+,-,+,-,-,-,-,-,+)&(4,4)&&(+,-,-,-,+,-,-,-,+)&(4,4)\\ 

&(+,-,-,-,-,-,-,-,+)&(2,6)&&(+,-,-,-,-,+,-,-,+)_2&(4,4)\\

&(+,-,-,-,-,-,-,+,+)_3&(2,6)&&(+,-,+,+,-,-,-,-,+)_4&(4,4)\\ 

&(+,-,+,-,-,-,-,+,+)_5&(4,4)&&(+,-,-,-,-,+,-,+,+)_6&(4,4)
\end{array}$$ 

$$\begin{array}{llllll}
(0,6)&(+,-,-,-,-,-,-,-,+)&(2,6)&&(+,+,-,-,-,-,-,+,+)&(2,6)\\ 

&(+,-,-,-,-,-,-,+,+)&(2,6)&&(+,+,+,+,-,-,-,-,+)&(2,6)\\ 

&(+,+,+,-,-,-,-,-,+)&(2,6)&&(+,+,+,-,-,-,-,+,+)_1&(2,6)

\end{array}$$ 

\begin{rems}\label{remarque1}
{\rm (1) When the sign pattern consists of a sequence of $m_1$ 
pluses followed by a sequence of $m_2$ minuses and then by a sequence of 
$m_3$ pluses, 
where $m_1+m_2+m_3=d+1$, then for the pair $(0,d-2)$ this sign pattern is 
not realizable if $\kappa :=(d-m_1-1)(d-m_3-1)/m_1m_3\geq 4$ (see Proposition~6 
in \cite{FoKoSh}). For the sign patterns with $(0,6)$ in the above table the 
quantity $\kappa$ equals respectively $36$, $25/4$, $15$, $9/2$, 
$8$ and $20/6<4$. The last inequality shows that Proposition~6 of \cite{FoKoSh} 
gives only sufficient, but not necessary conditions for nonrealizability of 
the pair $(0,d-2)$ with the sign patterns containing only two sign changes.

(2) In the problem which we consider an important role is played, although 
this is not always explicitly pointed out, by the {\em discriminant set} of the 
family of monic polynomials. This is the set of values of the coefficients 
for which the polynomial has a multiple root.  
The number of real roots changes, generically by $2$, 
when the tuple of coefficients crosses the 
discriminant set. The stratification of the 
discriminant set is explained in \cite{KhS}. More about discriminants of 
the general 
family of univariate polynomials for degree $4$ or $5$ can be found in 
\cite{Ko}.}
\end{rems} 

{\bf Acknowledgement.} The present paper is a continuation of the research 
on sign patterns and admissible pairs which was started by B.~Z.~Shapiro, 
J.~Forsg{\aa}rd and the author during the latter's stay at the University 
of Stockholm. The author expresses his most sincere gratitude to this 
university and to his former coauthors for this fruitful 
collaboration.

\section{Case 1 is not realizable\protect\label{seccase1}}

The proof that the sign pattern $\sigma _1$ 
is not realizable with the pair $(0,6)$ follows 
from Lemmas~\ref{RRRRbis} and~\ref{LLL}. The following lemma is 
used in the proof of Lemma~\ref{RRRRbis}.

\begin{lm}\label{RRRR}
For any $0<u<v$ there exists a polynomial 
$R=x^8+ax^7+bx^6+cx+d$, where $a>0$, $b>0$, $c>0$, $d>0$ and 
$R(-u)=R'(-u)=R(-v)=R'(-v)=0$. Hence by Descartes' rule of signs 
this polynomial equals $(x+u)^2(x+v)^2S(x)$, where the monic degree $4$ 
polynomial $S$ has no real roots.
\end{lm}

\begin{proof}
Consider the system of linear equations with unknown variables 
$a$, $b$, $c$ and $d$ and parameters $u>0$ and $v>0$:

$$\begin{array}{lll}
u^8-au^7+bu^6-cu+d=0&~~~~~&8u^7-7au^6+6bu^5-c=0\\
v^8-av^7+bv^6-cv+d=0&~~~~~&8v^7-7av^6+6bv^5-c=0~.\end{array}$$
One can solve this system w.r.t. $a$, $b$, $c$ and $d$ (using, say, MAPLE) 
and express the solutions as functions of $u$ and $v$. Set 

$$g:=35u^4v^4+20u^3v^5+4u^7v+10u^2v^6+u^8+4uv^7+20u^5v^3+v^8+10u^6v^2~.~~~
{\rm Then}$$

$$\begin{array}{lll}a&=&
(2/g)(u^9+4u^8v+10v^2u^7+20v^3u^6+35v^4u^5\\ \\ 
&&+35v^5u^4+20v^6u^3+10v^7u^2+4v^8u+
v^9)\\ \\ 
b&=&(1/g)(u^{10}+4vu^9+10u^8v^2+35u^4v^6\\ \\ 
&&+20u^7v^3+20u^3v^7+35u^6v^4+10u^2v^8+56u^5v^5+4uv^9+v^{10})\\ \\ 
c&=&(2u^5v^5/g)(5vu^4+6v^2u^3+6v^3u^2+3u^5+5v^4u+3v^5)\\ \\ 
d&=&(u^6v^6/g)(5u^4+8u^3v+9u^2v^2+8uv^3+5v^4)~.
\end{array}$$
All coefficients being positive, if one gives positive values to $u$ and $v$ 
($u\neq v$), one obtains positive values of $a$, $b$, $c$ and $d$.
\end{proof} 

\begin{lm}\label{RRRRbis}
If the sign pattern $\sigma _1$ is realizable with the pair 
$(0,6)$, then there exists a real monic degree $8$ polynomial $H$ 
having three double negative and one double positive root 
and the sign pattern $\sigma _1$.
\end{lm}

\begin{proof} 

Suppose that the sign pattern $\sigma _1$ 
is realizable with the 
pair $(0,6)$ by a real degree $8$ polynomial $P$ 
with six distinct negative roots and a complex conjugate pair. One can 
suppose that the values of $P$ at its negative critical points 
are all distinct.   
One can increase the constant term of $P$ (which does not change the sign 
pattern) so that two of the negative roots 
coalesce in a double negative root $\alpha$ which is a local minimum of $P$. 

Denote by $\tau <0$ and $\kappa <0$ the other two minima of $P$ on the negative 
half-axis (one has $P(\tau )<0$ and $P(\kappa )<0$). 

Denote by $R_1$ the polynomial of Lemma~\ref{RRRR} 
with $u=-\alpha$, $v=-\tau$. Then for $\varepsilon >0$ small 
enough the polynomial 
$T:=P+\varepsilon R_1$ has five distinct negative roots 
(four simple and one double). For some positive value of 
$\varepsilon =\varepsilon _0$ 
the polynomial $T$ has a double root at $\kappa$ as well. As the value of $T$  
for each fixed $x>0$ increases with $\varepsilon$, 
$T$ has no real positive root.  

Consider now the polynomial $T_0:=P+\varepsilon _0R_1$. 
Denote by $R_2$ the polynomial of Lemma~\ref{RRRR} 
with $u=-\alpha$, $v=-\kappa$. For some positive value of $\eta$ the polynomial 
$T^*:=T_0+\eta R_2$ has double roots at $\alpha$, $\kappa$ and $\tau$, 
no positive root and has the sign pattern $\sigma _1$.

Set $W:=(x-\alpha )^2(x-\kappa )^2(x-\tau )^2$. Consider the polynomial 
$T^*-\mu W$, $\mu >0$. All coefficients of $W$ are positive. 
Therefore the sign pattern defined by $T^*-\mu W$ has minuses in the 
positions in which $\sigma _1$ has such. 
As $T^*-\mu W$ has six negative roots counted with multiplicity, by  
Descartes' rule of signs the sign pattern defined by it has at most two 
sign changes. 

The polynomial $T^*-\mu W$ for $\mu >0$ small enough is of the form 
$(x-\alpha )^2(x-\kappa )^2(x-\tau )^2((x-\delta)^2+A)$, $\delta >0$, $A>0$. 
Indeed, if $\delta \leq 0$, then all coefficients of $T^*-\mu W$ 
would be positive and it will not define the sign pattern $\sigma _1$.

Decrease $A$. Denote by $\sigma '$ the sign pattern  
defined by $T^*-\mu W$ when $A=0$. When decreasing $A>0$, the signs of the 
coefficients of $x^j$ remain negative for $j=2$, $3$, $4$ and $5$. For 
$j=0$, $1$ and/or $6$ they might change from $+$ to $-$. 
If $\sigma '$ has more minuses than 
$\sigma _1$, 
then it has a sequence of $m_1$ pluses, $m_1\leq 3$, 
followed by a sequence of $m_2$ minuses followed by a sequence of $m_3$ pluses, 
$m_3\leq 2$, $m_1+m_2+m_3=9$ (because $T^*-\mu W$ has $6$ negative roots and 
the sequence of its coefficients must have at least $6$ sign 
preservations, i.e. not more than two sign changes). 

One cannot have $m_1<3$ or $m_3<2$ for $A=0$. Indeed, in this case one can 
increase slightly $A$ without changing $m_1$, $m_2$ and $m_3$ 
and obtain a contradiction with Proposition~6 of~\cite{FoKoSh}, 
see part (1) of Remarks~\ref{remarque1}. 
Hence $m_1=3$, $m_3=2$ and for $A=0$ the polynomial 
$T^*-\mu W$ defines the sign pattern $\sigma _1$, i.e. $\sigma '=\sigma _1$.
\end{proof}

\begin{lm}\label{LLL}
There exists no real monic degree $8$ polynomial having three double negative 
and one double positive root and defining the sign pattern $\sigma _1$.
\end{lm}

\begin{proof} 
Assume that such a polynomial exists. 
Without loss of generality one can assume that it is the square of the 
polynomial 

$$L:=(x^3+\alpha x^2+\beta x+\gamma )(x-1)=
x^4+(\alpha -1)x^3+(\beta -\alpha )x^2+(\gamma -\beta )x-\gamma$$ 
in which the first factor 
has three distinct negative roots. Hence $\alpha >0$, $\beta >0$ and 
$\gamma >0$. The coefficient of $x^s$ of $L^2$ is denoted by $c_s$. Hence 

\begin{equation}\label{ccc}
\begin{array}{ll}c_7=2(\alpha -1)&c_6=2(\beta -\alpha )+(\alpha -1)^2\\ 
c_5=2((\gamma -\beta )+(\alpha -1)(\beta -\alpha ))&c_2=(\gamma -\beta )^2-
2(\beta -\alpha )\gamma \\ 
c_1=-2\gamma (\gamma -\beta )&c_0=\gamma ^2~.
\end{array}
\end{equation}

\begin{rems}\label{remremrem}
{\rm (1) As $L^2$ defines the sign pattern $\sigma _1$, 
one must have $c_7>0$ and $c_1>0$ 
from which follows $\alpha >1$ and $\gamma <\beta$. These two inequalities 
combined with $c_2<0$ yield $\beta >\alpha$.

(2) The condition $\beta >\alpha$ implies that the absolute value of 
at least one of the roots of the polynomial $x^3+\alpha x^2+\beta x+\gamma$ 
(which are all negative) is $>1$.}
\end{rems}

In what follows we denote by $\mathcal{P}$ the set 
$\{ \alpha >1,\beta >0,\gamma >0\}$. For each $\alpha =\alpha _0>1$ fixed 
the set $\mathcal{P}|_{\alpha =\alpha _0}$ is the positive quadrant 
$\{ \beta >0,\gamma >0\}$. 

\begin{lm}\label{severalsets}
Suppose that $\alpha =\alpha _0>1$ is fixed. Then:

(1) The condition 
$c_5=0$ defines a straight line $\mathcal{C}_5$. 
Its slope $2-\alpha _0$ is positive for $\alpha _0\in (1,2)$, 
zero for $\alpha _0=2$ and negative for $\alpha _0>2$. For $\alpha _0>2$ 
the intersection $(\mathcal{P}|_{\alpha =\alpha _0})\cap \mathcal{C}_5$ 
is a segment. 

(2) The condition $c_2=0$ defines 
a hyperbola with centre $(2\alpha _0/3,\alpha _0/3)$
and with asymptotes $\gamma -\alpha _0/3=(2\pm \sqrt{3})(\beta -2\alpha _0/3)$. 
One of its branches (denoted by $\mathcal{C}_2$) 
belongs to the set $\mathcal{P}|_{\alpha =\alpha _0}$; the other one 
is denoted by $\mathcal{C}_2^*$. The point $(0,0)$ belongs to 
$\mathcal{C}_2^*$ and the tangent line to $\mathcal{C}_2^*$ at $(0,0)$ is 
horizontal. Hence 
$\mathcal{C}_2^*\cap (\mathcal{P}|_{\alpha =\alpha _0})=\emptyset$.  

(3) For $\alpha _0>\sqrt{3}$ 
the intersection $\mathcal{C}_5\cap \mathcal{C}_2$ consists of the two 
points 

$$I_1:=(\alpha _0,\alpha _0)~~~~{\it and}~~~~
I_2:=(\alpha _0(\alpha _0^2-1)/(\alpha _0^2-3),
\alpha _0(\alpha _0-1)^2/(\alpha _0^2-3))~.$$
For $\alpha _0\in (1,\sqrt{3}]$ one has $\mathcal{C}_5\cap \mathcal{C}_2=I_1$. 
The tangent line to $\mathcal{C}_2$ at $I_1$ is vertical, at $I_2$ its slope 
is negative for $\alpha _0>3$, zero for $\alpha _0=3$ and positive 
for $\alpha _0\in (1,3)$. For $\alpha _0>3$ this slope is negative for 
the points of $\mathcal{C}_2$ which are between $I_1$ and $I_2$.

(4) The set of hyperbolic polynomials is defined by the condition 

\begin{equation}\label{conditionhyperbolicity}
4(\beta -\alpha _0^2/3)^3+27(\gamma +2\alpha _0^3/27-
\alpha _0\beta /3)^2\leq 0~.
\end{equation}
The corresponding equality defines a curve $\mathcal{H}$ 
having as only singular point a cusp at 
$J:=(\alpha _0^2/3,\alpha _0^3/27)$. The set of hyperbolic polynomials 
is the closure of the interior of $\mathcal{H}$. 
The slope of the tangent lines to $\mathcal{H}$ 
at its regular points (and the one of 
the geometric semi-tangent at its cusp) is positive for 
$\beta >0$, $\gamma >0$. The maximal values of the coordinates of the 
restriction of $\mathcal{H}$ to $\{ \beta >0,\gamma >0\}$ are attained, 
simultaniously for $\beta$ and $\gamma$, at and only at its cusp. 

(5) The curve $\mathcal{H}$ intersects the line $\mathcal{C}_5$ exactly when 
$\alpha _0\geq u_0:=3.787042615\ldots$. For $\alpha _0<u_0$ the cusp point 
$J$ lies below the line $\mathcal{C}_5$. The point $I_2$ 
does not define a hyperbolic polynomial for any $\alpha _0>1$.
\end{lm}

\begin{figure}[htbp]
\centerline{\hbox{\includegraphics[scale=0.7]{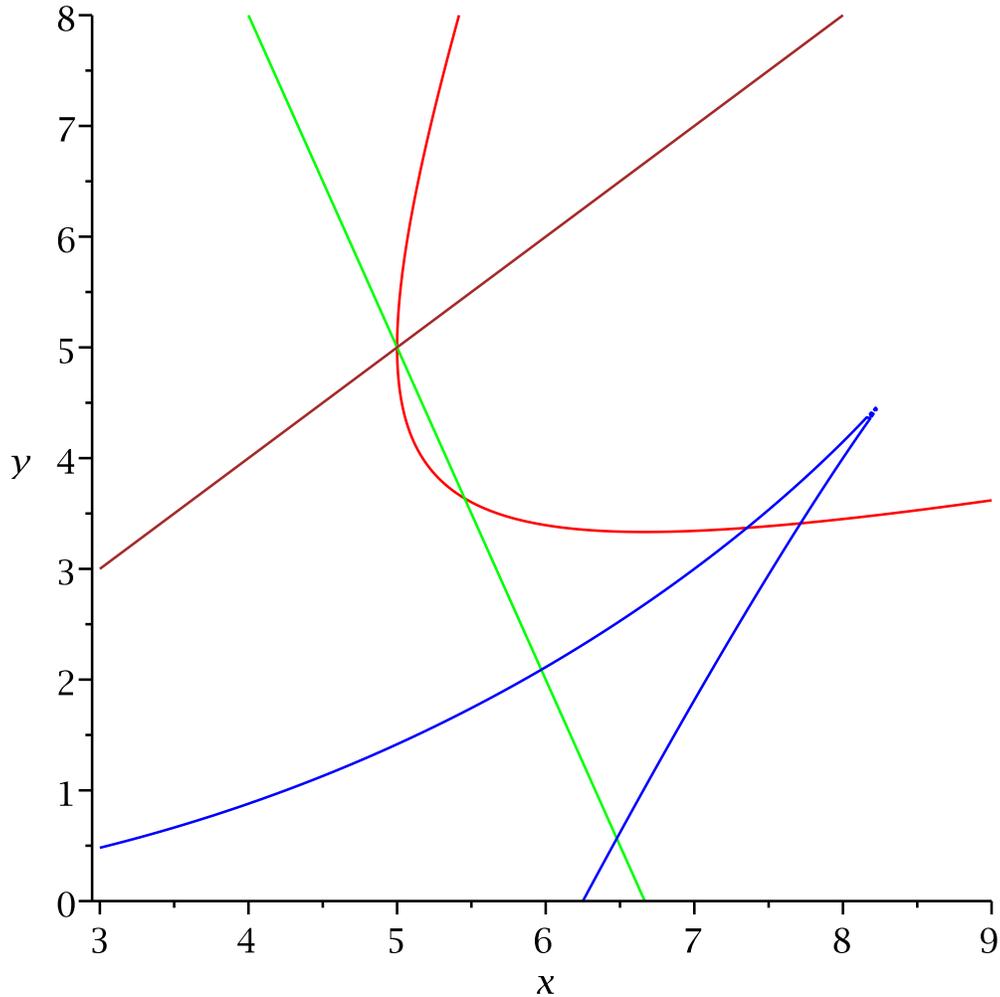}}}
    \caption{The sets $\mathcal{C}_2$, $\mathcal{C}_5$, 
$\{ \beta =\gamma \}$ and $\mathcal{H}$.}
\label{c2c5dd}
\end{figure}

Before proving Lemma~\ref{severalsets} we finish the proof of Lemma~\ref{LLL}. 
On Fig~\ref{c2c5dd} we show the sets $\mathcal{C}_2$ (branch of a hyperbola), 
$\mathcal{C}_5$ (straight line with negative slope), the straight line  
$\{ \beta =\gamma \}$ and $\mathcal{H}$ (curve with a cusp point) 
for $\alpha _0=5$. The set $\{ c_2<0\}$ 
is the interior of the branch $\mathcal{C}_2$ and the set $\{ c_2<0,c_5<0\}$ 
is the lens-shaped domain between $\mathcal{C}_2$ and 
$\mathcal{C}_5$. The point $I_1$ is the triple intersection of 
$\mathcal{C}_2$, 
$\mathcal{C}_5$ and 
$\{ \beta =\gamma \}$.

\begin{rem}
{\rm For $\alpha _0\in (1,\sqrt{3}]$ the set $\{ c_2<0,c_5<0\}$ is not compact 
and for $\alpha _0\in (1,\sqrt{3})$ the point $I_2$ 
belongs not to $\mathcal{C}_2$, but to 
$\mathcal{C}_2^*$; $I_2$ is at $\infty$ for $\alpha _0=\sqrt{3}$. 
Indeed, the slopes of the asymptotes of 
the hyperbola $\{ c_2=0\}$ equal $2\pm \sqrt{3}$ while the slope of 
$\mathcal{C}_5$ equals $2-\alpha _0$,  
see parts (1) and (2) of Lemma~\ref{severalsets}.}
\end{rem}

There exists a unique point $Z\in \mathcal{C}_2$ the tangent to $\mathcal{C}_2$ 
at which is horizontal. Indeed, the branches $\mathcal{C}_2$ and 
$\mathcal{C}_2^*$ of the hyperbola $\{ c_2=0\}$ are symmetric w.r.t. its centre 
$(2\alpha _0/3,\alpha _0/3)$, see part (2) of Lemma~\ref{severalsets}. The only 
point of $\mathcal{C}_2^*$ at which the tangent line is horizontal is the 
origin, see part (2) of Lemma~\ref{severalsets} (the fact that $(0,0)$ is the 
only such point follows from the convexity of the hyperbola). 
Hence $Z=(4\alpha _0/3,2\alpha _0/3)$. 

Compare the $\gamma$-coordinates of the points $Z$ and $J$ (see part (4) 
of Lemma~\ref{severalsets}). For $\alpha _0<3\sqrt{2}=4.2\ldots$ one has 
$2\alpha _0/3>\alpha _0^3/27$. The point $Z$ has the least possible 
$\gamma$-coordinate of the points of $\mathcal{C}_2$ whereas $J$ has the 
largest possible $\gamma$-coordinate of the points of 
$\mathcal{H}\cap \mathcal{P}_{\alpha =\alpha _0}$, see part (4) of 
Lemma~\ref{severalsets}. Hence for $\alpha _0\in (1,3\sqrt{2})$ one has 

$$\mathcal{C}_2\cap (\mathcal{H}\cap \mathcal{P}_{\alpha =\alpha _0})=\emptyset 
~~~\, \, {\rm and}~~~\, \,  
\{ c_2<0,c_5<0\} \cap (\mathcal{H}\cap \mathcal{P}_{\alpha =\alpha _0})=
\emptyset ~.$$
Recall that $u_0<3\sqrt{2}$, see part (5) of Lemma~\ref{severalsets}. 
Hence for $\alpha _0=u_0$ the cusp $J$ of $\mathcal{H}$ has a smaller 
$\gamma$-coordinate than $I_2$. As $I_2$ does not belong to $\mathcal{H}$ 
(for any $\alpha _0>1$, see part (5) of Lemma~\ref{severalsets}), for 
$\alpha _0>u_0$ the points $I_1$ and $I_2$ are above the two intersection points 
$K_1$ and $K_2$ of $\mathcal{H}$ with $\mathcal{C}_5$ (``above'' 
means ``have larger $\gamma$-coordinates''); $K_1$ is presumed to be 
above $K_2$. Denote by $L^*$ and $L^{**}$ 
the vertical straight lines passing through $I_2$ and $K_1$. Hence for $a>3$ 
the domain $\{ c_2<0,c_5<0\}$ lies to the left of $L^*$ and above $I_2$, see 
parts (1) and (3) of Lemma~\ref{severalsets}. At the same time the part of 
$\mathcal{H}\cap \mathcal{P}_{\alpha =\alpha _0}$ which is to the left of 
$L^*$ (hence to the left of $L^{**}$ as well) lies below $K_1$ hence below 
$I_2$, so the domain $\{ c_2<0,c_5<0\}$ contains no hyperbolic polynomial. 
This proves Lemma~\ref{LLL} and Theorem~\ref{tmnew}.
\end{proof}

\begin{proof}[Proof of Lemma~\ref{severalsets}]
The first two statements of part (1) are to be checked directly. To prove the 
third statement it suffices to compute the intersection points of the line 
$\mathcal{C}_5$ with the $\beta$- and $\gamma$-axes. These points are 
$(0,\alpha _0(\alpha _0-1))$ and 
$(\alpha _0(\alpha _0-1)/(\alpha _0-2),0)$.

Prove part (2). The determinants of the matrices $M_1=\left( \begin{array}{rrr}
1&-2&0\\ -2&1&\alpha _0\\ 0&\alpha _0&0\end{array}\right)$ and 
$M_2=\left( \begin{array}{rr}
1&-2\\ -2&1\end{array}\right)$ 
(defined after the quadric $c_2|_{\alpha =\alpha _0}$) 
are nonzero and $M_2$ has one 
positive and one negative eigenvalue. Hence the equation $c_2=0$ defines 
a hyperbola.
To find its centre one sets $\beta \mapsto \beta +\mu$, 
$\gamma \mapsto \gamma +\nu$ and one looks for $(\mu ,\nu )$ such that the 
linear terms in the equation $c_2=0$ disappear. This yields the system

$$ -4\mu +2\nu +2\alpha _0=0~~~~,~~~~2\mu-4\nu =0$$
whose solution is $(\mu ,\nu )=(2\alpha _0/3,\alpha _0/3)$. The slopes of the 
asymptotes are solutions to the equation $\lambda ^2-4\lambda +1=0$ deduced 
from the matrix $M_2$. The branch 
$\mathcal{S}_2$ occupies the upper right sector defined by the asymptotes.     

The equation $c_2=0$ is satisfied for $(\beta ,\gamma )=(0,0)$. To compute the 
equation of the tangent line to the hyperbola $\{ c_2=0\}$ one writes 

\begin{equation}\label{tangent}
(-4\beta +2\gamma +2\alpha )d\gamma +(2\beta -4\gamma )d\beta =0
\end{equation}
in which the coefficient of $d\beta$ is $0$ for $(\beta ,\gamma )=(0,0)$. 
The tangent at $(0,0)$ being horizontal the branch $\mathcal{S}_2^*$ belongs 
entirely to the lower half-plane and does not intersect the set 
$\mathcal{P}|_{\alpha =\alpha _0}$.  

Prove part (3). 
Set $B:=\gamma -\beta$, $A:=\beta -\alpha _0$. The conditions $c_5=0$ and 
$c_2=0$ read (see (\ref{ccc})):

$$B=-(\alpha _0-1)A~~,~~-2(B+A+\alpha _0)A+B^2=0$$
from which one finds that either $A=0$ (hence $B=0$ and 
$\beta =\gamma =\alpha _0$, this defines the point $I_1$) 
or $-2(-(\alpha _0-1)A+A+\alpha _0)+(\alpha _0-1)^2A=0$. 
The last equality implies 
$A=2\alpha _0/(\alpha _0^2-3)$. Hence  

\begin{equation}\label{firstequa}
\beta =\alpha _0(\alpha _0^2-1)/(\alpha _0^2-3)~,
\end{equation} 
so $B=2\alpha _0(1-\alpha _0)/(\alpha _0^2-3)$ and 

\begin{equation}\label{secondequa}
\gamma =\alpha _0(\alpha _0-1)^2/(\alpha _0^2-3)~.
\end{equation} 
which gives the point $I_2$. To show that the tangent line to $\mathcal{C}_2$ 
at $I_1$ is vertical it suffices to observe that for $\beta =\gamma =\alpha _0$ 
equation (\ref{tangent}) reduces to $d\beta =0$. At $I_2$ the tangent line 
to $\mathcal{C}_2$ is defined by the equation 

$$(2\alpha _0^2/(\alpha _0^2-3))d\gamma +
(\alpha _0(\alpha _0-1)(\alpha _0-3)/(\alpha _0^2-3))d\beta =0~.$$
Its slope equals $-(\alpha _0-1)(\alpha _0-3)/2\alpha _0$. The last statement 
of part (3) follows from the convexity of the hyperbola $\{ c_2=0\}$.   

To prove part (4) one has to recall that the real polynomial 
$x^3+px+q$ is hyperbolic if and only 
if $4p^3+27q^2\leq 0$ (this means, in particular, that $p\leq 0$). As 

$$x^3+\alpha x^2+\beta x+\gamma =(x+\alpha /3)^3+
(\beta -\alpha ^2/3)(x+\alpha /3)+\gamma +2\alpha ^3/27-\alpha \beta /3~,$$
the polynomial $L|_{\alpha =\alpha _0}$ is hyperbolic if and only if 
condition (\ref{conditionhyperbolicity}) holds true. 

Set $\beta \mapsto \alpha _0^2\beta$ and $\gamma \mapsto \alpha _0^3\gamma$ 
in the equation of $\mathcal{H}$ (see (\ref{conditionhyperbolicity})). 
In the new variables 
$(\beta ,\gamma )$ the equation of $\mathcal{H}$ (after division by 
$\alpha _0^3$) coincides with its equation 
for $\alpha _0=1$:

\begin{equation}\label{ch1}
4(\beta -1/3)^3+27(\gamma +2/27-\beta /3)^2=0~.
\end{equation}
One can parametrize this curve by setting $\beta =1/3-3t^2$, 
$\gamma =1/27+2t^3-t^2=2(t-1/3)^2(t+1/6)$. 
It has a cusp for $t=0$, i.e. at $(1/3,1/27)$. 
Its tangent vector equals $(-6t,6t^2-2t)$. For $t<0$ its components are 
both positive and its slope is also positive. For $t\in (0,1/3)$ they are 
both negative and again the slope is positive. One has $\beta >0$ and 
$\gamma >0$ exactly when $t\in (-1/6,1/3)$ (i.e. only for values of 
$t$ for which the slope is positive). The coordinate $\beta$ attains its 
global maximal value $1/3$ only for $t=0$. For $t\in (-1/6,1/3)$ the coordinate 
$\gamma$ attains its maximal value $1/27$ only for $t=0$. 

Prove part (5). The equation of $\mathcal{H}$ with 
$\gamma =\beta -(\alpha _0-1)(\beta -\alpha _0)$ reads:

\begin{equation}\label{Uu}
\begin{array}{ccl}\mathcal{U}(\alpha _0,\beta )&:=&
4\beta ^3+44\beta ^2\alpha _0^2-4\beta \alpha _0^4-108\alpha _0\beta 
+27\alpha _0^2-54\alpha _0^3\\ && 
+108\beta ^2+180\beta \alpha _0^2-
144\alpha _0\beta ^2-64\beta \alpha _0^3+23\alpha _0^4+4\alpha _0^5=0~.
\end{array}
\end{equation}  
One has 

$${\rm Res}(\mathcal{U},\partial \mathcal{U}/\partial \beta ,\beta )=
-64\alpha _0^3(\alpha _0-1)(2\alpha _0^2-7\alpha _0+8)(10\alpha _0^2-
45\alpha -0+27)^3~.$$
The first quadratic factor has no real roots. The roots of the second 
one equal $0.7129573851\ldots <1$ and $u_0:=3.787042615\ldots$. 
For $\alpha _0=u_0$ the cusp point of $\mathcal{H}$ is on $\mathcal{C}_5$. 
For $\alpha _0<u_0$ the curve $\mathcal{H}\cap \mathcal{P}|_{\alpha =\alpha _0}$ 
lies entirely below the line $\mathcal{C}_5$ (this can be deduced from the 
last statement of part (4) of the lemma and from the fact that for 
$\alpha _0>0$ small enough the cusp point $J$ is close to the origin); 
for $\alpha _0>u_0$ it intersects this line at two points. 

\begin{rem}
{\rm Equation (\ref{Uu}) is of degree $3$ w.r.t. $\beta$. On Fig.~\ref{c2c5dd} 
one sees two of the solutions (the points $K_1$ and $K_2$, see the proof 
of Lemma~\ref{LLL}). The third solution is an intersection point of 
$\mathcal{H}$ with $\mathcal{C}_5$, with $\beta <0$ and $\gamma >0$. Such 
an intersection point exists because the $\gamma$-coordinate of a point of 
$\mathcal{C}_5$ grows linearly in $|\beta |$ as $|\beta |$ increases 
($\beta$ being negative) while the $\gamma$-coordinate of a point of 
$\mathcal{H}$ grows as $|\beta |^{3/2}$.}
\end{rem}

To prove the last statement of part (5) we substitute the right-hand sides of 
(\ref{firstequa}) and (\ref{secondequa}) 
for $\beta$ and $\gamma$ 
in (\ref{conditionhyperbolicity}) and we multiply by 
$(\alpha _0^2-3)^3/\alpha _0^2(\alpha _0-1)^2>0$. This yields   
the equivalent condition 

$$3\alpha _0^6-16\alpha _0^5+13\alpha _0^4+24\alpha _0^3-23\alpha _0^2+
104\alpha _0-81\leq 0~.$$ 
However the left-hand side has no 
roots greater than $1$ and the leading coefficient is positive. Hence 
the last inequality fails for $\alpha _0>1$. 

\end{proof}


\section{Cases 2 - 6 are not realizable\protect\label{seccases26}}

\subsection{Preliminaries}
The following two lemmas are proved in the Appendix. 
They allow to simplify 
the proof of Theorem~\ref{tmd8} by decreasing the number of parameters.

\begin{lm}\label{basiclemma}
Suppose that there exists a monic degree $8$ polynomial $P$ realizing 
Case $j$, $2\leq j\leq 6$. Then there exists a monic degree $8$ polynomial 
$U$ having a quadruple root at $1$ and no other real roots, and whose 
coefficients define the same sign pattern as the one of Case~$j$.
\end{lm}

\begin{rem}
{\rm One can observe that roots at $1$ remain invariant under reverting of 
sign patterns.}
\end{rem}

\begin{lm}\label{nextlemma}
(1) Suppose that a monic polynomial $U=(x-1)^4V$ realizes one of 
the sign patterns 

$$\begin{array}{lllllll}
\sigma _2^r&=&(+,+,-,-,-,+,-,+,+)&~~~~~~~~~&  
\sigma _4&=&(+,+,+,-,-,+,-,+,+)\\
&&~~~~~{\rm or}&&\sigma _6^r&=&(+,-,-,-,-,+,-,+,+)~,
\end{array}$$
where $V$ is a real monic polynomial with no real root. 
Then there exists a polynomial of the form  $U_t:=U-t(x-1)^4$, $t\geq 0$, 
defining the same sign pattern and having one or two negative roots 
of even multiplicity, hence a polynomial of the form 

\begin{equation}\label{polyW}
W:=(x-1)^4(x^2+Sx+S^2/4)(x^2+Mx+N)~,~~{\rm where}~~S>0~~{\rm and}~~
N\geq M^2/4~.
\end{equation} 

(2) If the polynomial $U$ realizes the sign pattern 
$\sigma _3^r=(+,-,-,+,-,+,-,+,+)$, then in the family of polynomials 
$U^*_t:=U+tx(x-1)^4$, $t>0$, there exists a polynomial 
defining the sign pattern $\sigma _3^r$ and of the form (\ref{polyW}).

(3) If the polynomial $U$ realizes the sign pattern 
$\sigma _5^r=(+,-,-,+,-,+,+,+,+)$, then in the family of polynomials 
$U^*_t:=U+tx(x-1)^4$, $t>0$, there exists a polynomial defining 
one of the sign patterns 
$\sigma _3^r$, $\sigma _5^r$ or $\sigma ^*:=(+,-,-,+,-,+,0,+,+)$ 
and of the form (\ref{polyW}).
\end{lm}

In what follows we set $W:=\sum _{j=0}^8w_jx^j$, $w_8=1$, and  

$$Q:=3S^2/2-4S+1~~,~~R:=S^2-6S+4~~{\rm and}~~P:=S^2/4-4S+6~.$$
The roots of these three polynomials are real. We denote them by 

$$\begin{array}{ccccccccccc}
0.27\ldots &=&(4-\sqrt{10})/3&=&q_1&<&q_2&=&(4+\sqrt{10})/3&=&
2.38\ldots \\ 
0.76\ldots &=&3-\sqrt{5}&=&r_1&<&r_2&=&3+\sqrt{5}&=&
5.23\ldots \\ 
1.67\ldots &=&8-\sqrt{40}&=&p_1&<&p_2&=&8+\sqrt{40}&=&
14.32\ldots \end{array}$$
The coefficients 
$w_j$, $j=0$, $\ldots$, $7$ are expressed by the following formulae:

\begin{equation}\label{wj}\begin{array}{ccccccc}
w_0&=&S^2N/4&&w_1&=&(S/4)(MS+4N(1-S))\\
w_2&=&QN+(S-S^2)M+S^2/4&&w_3&=&QM-RN+S-S^2\\
w_4&=&Q-RM+PN&&w_5&=&-R+PM+(S-4)N\\ 
w_6&=&P+(S-4)M+N&&w_7&=&M+S-4
\end{array}
\end{equation}

\subsection{Cases 2 and 4}

In Cases 2 and 4 we are using the sign patterns  
$\sigma _2^r=(+,+,-,-,-,+,-,+,+)$ 
and $\sigma _4=(+,+,+,-,-,+,-,+,+)$. They can be united 
in a single sign pattern $\pi _{\pm}:=(+,+,\pm ,-,-,+,-,+,+)$. 
If the polynomial $W$ (see (\ref{polyW}) defines 
the sign pattern $\pi _{\pm}$, then one must have $w_j>0$ for $j=0$, $1$, $3$ 
and $7$ and $w_j<0$ for $j=2$, $4$ and $5$.

One has $M>0$. Indeed, $w_7=M+S-4>0$, hence $S>4-M$. Suppose that $M\leq 0$. 
Then one has $S>4$, $MS\leq 0$ 
and $4N(1-S)\leq 0$, i.e. $w_1\leq 0$ -- a contradiction. 


Suppose that $S>1$. Then the condition $w_1>0$ is equivalent to 
$N<MS/4(S-1)$. On the other hand, as $N\geq M^2/4$, the last two inequalities 
together imply $M<S/(S-1)$ hence $N<S^2/u$, where $u=4(S-1)^2$. 

For $S\in [p_1,p_2]$ (recall that $p_1>1$) one has 
$P\leq 0$, $PN\geq PS^2/u$ and $4-S<M<S/(S-1)$. 
Therefore 

$$w_4\geq \min 
(~Q(S)-R(S)(4-S)+P(S)S^2/u~,~Q(S)-R(S)S/(S-1)+P(S)S^2/u~)~.$$
This minimum is $>5$ hence $>0$ (the numerical check of this is easy) 
and the inequality $w_4<0$ 
fails for $S\in [p_1,p_2]$.

For $S>p_2$  
one has $P\geq 0$, $PN\geq PM^2/4\geq 0$ and $0<M<S/(S-1)$, so 

$$w_4\geq \min 
(~Q(S)-R(S)S/(S-1)~,~Q(S)~)~.$$
This minimum is also positive and again $w_4<0$ fails. 

Let now $S\in (0,p_1)$.  
The inequality $w_4<0$ can be rewritten as 
$N<(RM-Q)/P$ which together with $M^2/4\leq N$ implies 
$PM^2-4RM+4Q<0$. This is a quadratic inequality w.r.t. $M$, with $P>0$. 
The discriminant 
of the quadratic polynomial $Y(M,S):=P(S)M^2-4R(S)M+4Q(S)$ equals 
$4(R^2(S)-P(S)Q(S))$. 
It is positive for all 
$S\in (0,p_1)$ (this is easy to check). Hence for $S\in (0,p_1)$ the polynomial 
$Y$ has two real roots $M'<M''$ which depend continuously on $S$ and 
one must have $M\in (M',M'')$. 

For each $S\in (0,p_1)$ fixed both these roots are smaller than $4-S$. Indeed, 
set $M:=4-S$. The polynomial $Y(4-S,S)$ is positive on $(0,p_1)$ 
(easy to check). For $S=1\in (0,p_1)$ one has $Q=-3/2<0$, i.e. 
one of the roots is 
negative and the other is positive. Hence for $S\in (0,p_1)$ the number 
$4-S$ lies 
outside the interval $[M',M'']$, and as $4-S>0$, one has $M'<4-S$,  
$M''<4-S$ and $M\in (M',M'')$. But one must have $M>4-S$, 
so the inequalities $w_7>0$ 
and $w_4<0$ cannot hold simultaneously for $S\in (0,p_1)$.

\subsection{Cases 3, 5 and 6}

In Cases 3, 5 and 6 we use the sign patterns 

$$\begin{array}{lllllll}\sigma _3^r&=&(+,-,-,+,-,+,-,+,+)&~~~~,& 
\sigma _5^r&=&(+,-,-,+,-,+,+,+,+)\\ &&{\rm and}&&   
\sigma _6^r&=&(+,-,-,-,-,+,-,+,+)~.\end{array}$$
and formulae (\ref{wj}). The proof of Theorem~\ref{tmnew} in these cases 
results from Lemmas~\ref{not356r1}, \ref{notp14} and~\ref{notr1p1}. 

\begin{lm}\label{M>0}
In Cases 3, 5 and 6 one has $M>0$.
\end{lm}

\begin{proof}
One  must have $w_1>0$ and $w_6<0$. 
For $S\geq 1$ the product $N(1-S)$ is negative (see formulae 
(\ref{wj})), so for $S\geq 1$ the condition $w_1>0$ implies that one must have 
$SM>0$, i.e. $M>0$. Consider for $S\in (0,1)$ the condition $w_6<0$, (i.e.  
$P+(S-4)M+N<0$).    
One has $P(S)>0$, $N\geq 0$ and $S-4<0$, so the inequality $w_6<0$ is possible 
only for $M>0$. 
\end{proof}

\begin{lm}\label{not356r1}
Cases 3, 5 and 6 are not realizable for $S\in (0,r_1]$.
\end{lm}

\begin{proof}
In Cases 3, 5 and 6 one has $w_3>0$, i.e. $QM+S-S^2>RN$, see (\ref{wj}). 
For $S\in (0,r_1]$ one has $R(S)\geq 0$ and 
$QM+S-S^2>RN\geq RM^2/4$, hence 

\begin{equation}\label{aster}
L(S,M):=R(S)M^2/4-Q(S)M-S+S^2<0~.
\end{equation} 
The inequalities 
(\ref{aster}), $0<S\leq r_1$ and $0\leq M<4-S$ have no common solution. Indeed, 
$L(S,4-S)=(S-2)^2((S-2)^2+8)/4$. This means that for $S=2$ 
the line $M+S=4$ has an 
ordinary tangency with the curve $L(S,M)=0$, and this is 
their only common point in the domain $\{ S>0,M>0\}$. For $S=M=1/2$ one has 
$L(S,M)=9/64>0$ and $S+M-4<0$. Hence 
below the line $M+S=4$ in the domain $\{ S>0,M>0\}$ one has $L(S,M)>0$.
\end{proof}

\begin{rem}
{\rm (1) The inequalities $S>0$, $M>0$ (see Lemma~\ref{M>0}) and 
$S+M<4$ (this follows from $w_7<0$ in Cases 3, 5 and 6) imply $S<4$.}
\end{rem}

\begin{convention}\label{conv1}
{\rm (1) In what follows we interpret an equality of the form $w_j=0$ 
(see (\ref{wj})) as the equation of a straight line (denoted by $\ell _j$) 
in the space $(M,N)$ 
with coefficients depending on $S$ as on a parameter. Most often 
we need equations  
of the form $A(S)N+B(S)M+C(S)=0$, and we care to have a positive coefficient 
of $N$. E.g. we prefer the equation of the line $\ell _1$ 
(see the quantity $w_1$ in 
formulae (\ref{wj})) to be of the form $4(1-S)N+SM=0$ for $S<1$ and 
$4(S-1)N-SM=0$ for $S>1$. 

(2) We denote by $\ell _j^+$ (resp. $\ell _j^-$) the 
upper (resp. lower) half-plane defined by the line $\ell _j$. In the case 
of $\ell _1$ one has $\ell _1^+:4(1-S)N+SM>0$ for $S<1$ and 
$\ell _1^+:4(S-1)N-SM>0$ for $S>1$. For $S=1$ this line is vertical 
and we do not define 
the half-planes $\ell _1^{\pm}$. By $s(\ell _j)$ we denote the {\em slope} of 
the line $\ell _j$, i.e. the quantity $-B(S)/A(S)$ for $A(S)\neq 0$. 
For $\ell _1$ it equals $S/4(S-1)$.  

(3) When in the proofs of the lemmas rational functions appear, it is 
presumed that the factors of degree $2$ have no real roots (so 
their sign coincides with the one of their leading coefficient). Factorizations 
are performed by means of MAPLE.}
\end{convention} 

\begin{lm}\label{notp14}
Cases 3, 5 and 6 are not realizable for $S\in [p_1,4)$.
\end{lm}

\begin{proof}
Consider the four conditions $M>0$, $w_1>0$, $w_3>0$ and $w_4<0$. 
The second of them defines the half-plane $\ell _1^-$ 
(recall that $\ell _1:4(S-1)N-SM=0$). The last two of them 
read

$$(-R(S))N+Q(S)M+S-S^2>0\hspace{6mm}{\rm and}\hspace{6mm}
(-P(S))N+R(S)M-Q(S)>0~.$$
The straight line $\ell _3:(-R(S))N+Q(S)M+S-S^2=0$ intersects the $N$-axis 
at the point $A:=(0,N_A)$ with $N_A:=S(S-1)/(-R(S))>0$. The lines 
$\ell _3$ and $\ell _4:(-P(S))N+R(S)M-Q(S)=0$ intersect at the point $B$ 
with coordinates 
$$\begin{array}{lllll}
M_B&:=&(2/5)(5S^4-35S^3+84S^2-64S+16)/K(S)&,&\\ 
N_B&:=&(2/5)(5S^4-20S^3+36S^2-16S+4)/K(S)&,&{\rm where}\\
K(S)&:=&S^4-8S^3+30S^2-32S+16&.\end{array}$$
and both numerators and the denominator $K$ have no real roots.
This point lies above the straight line $\ell _1$. Indeed, the coefficient 
of $N$ in the equation of $\ell _1$ is positive. Substituting 
$(M_B,N_B)$ for $(M,N)$ in the left-hand side of this equation  
yields the expression 

$$\mu :=\frac{6(S^2-2.5\ldots S+3.8\ldots )(S^2-0.5\ldots S+0.2\ldots )
(S-1.2\ldots )}{(S^2-6.6\ldots S+20.4\ldots )(S^2-1.3\ldots S+0.7\ldots )}~$$
which is positive, see Convention~\ref{conv1}.

For the slopes $s(\ell _4)$ and $s(\ell _1)$ one has  
$s(\ell _4)>s(\ell _1)>0$. 
The first inequality follows from $R(S)/P(S)-S/4(S-1)>0$ which is equivalent to 

$$\frac{15(S^2-1.7\ldots S+0.9\ldots )(S-4.6\ldots )}{4P(S)(S-1)}>0$$
and this results from $S-4.6\ldots <0$, $S-1>0$ and $P(S)<0$. 

Hence the set defined by the conditions $M>0$, $w_3>0$ and $w_4<0$ is the 
domain of $\mathbb{R}^2\simeq (M,N)$ to the right of the $N$-axis, to the 
above of the segment $AB$ and to the above of the half-line starting at $B$,  
which is part of the line $\ell _4$ and which goes to the right and upward. 
This domain does not intersect the half-plane $\ell _1^-$ and the four 
conditions $M>0$, $w_1>0$, $w_3>0$ and $w_4<0$ cannot hold true simultaneously.
\end{proof}


\begin{lm}\label{notr1p1}
Cases 3, 5 and 6 are not realizable for $S\in (r_1,p_1)$.
\end{lm}

\begin{proof}
Consider the conditions $w_3>0$ and $w_6<0$. They read 

$$(-R(S))N+Q(S)M+S-S^2>0\hspace{6mm}{\rm and}\hspace{6mm}
N+(S-4)M+P(S)<0~.$$
Consider the point $\Pi :=\ell _3\cap \ell _6$. Its coordinates equal

$$(-(S^4-22S^3+120S^2-204S+96)/2Y(S),-3(S^4-16S^3+54S^2-64S+16)/4Y(S))~,$$
where $Y(S):=2S^3-17S^2+48S-30$ has a single real root 
$y_0:=0.8609094817\ldots$. For $S\in (r_1,y_0)$ (resp. for  $S\in (y_0,p_1)$) 
one has $s(\ell _3)>s(\ell _6)$ (resp. $s(\ell _3)<s(\ell _6)$). 
This follows from 

$$Q(S)/R(S)-(4-S)=(S^2-7.6\ldots S+17.4\ldots )(S-y_0)/R(S)$$
with $R(S)<0$. The second coordinate of $\Pi$ equals

$$-3(S-0.3\ldots )(S-11.9\ldots )(S^2-3.7\ldots S+4.0\ldots )/4Y(S)~.$$
Hence it changes sign from $-$ to $+$ when $S$ passes from $y_0^-$ to $y_0^+$. 
For $S\in (r_1,y_0)$ one has 
$\{ w_3>0\} \cap \{ w_6<0\} =\ell _3^+\cap \ell _6^-$. For $S=y_0$ the 
lines $\ell _3$ and $\ell _6$ are parallel, $\ell _3$ is 
above $\ell _6$ and $\{ w_3>0\} \cap \{ w_6<0\} =\emptyset$. 
Thus for $S\in (r_1,y_0)$ the sector $\ell _3\cap \ell _6$ belongs to 
the domain $N<0$ and if some of Cases 3, 5 or 6 is realizable, it can be 
realizable only for $S\in (y_0,p_1)$. 

For $S=y_0^+$ the intersection $\{ w_3>0\} \cap \{ w_6<0\}$ is a sector whose 
vertex has both coordinates positive because 
the first coordinate of $\Pi$ equals

$$-(S-0.7\ldots )(S-1.8\ldots )(S-4.6\ldots )(S-14.7\ldots )/2Y(S)>0~.$$
The point $\Pi$ lies above the line $\ell _4:P(S)N-R(S)M+Q(S)=0$ for 
$S\in (y_0,y_1)$, where $y_1:=1.471576286\ldots$. 
Indeed, substituting the coordinates of $\Pi$ for $(M,N)$ in the left-hand 
side of the equation of $\ell _4$ yields

$$ \frac{5(S^2-5.2\ldots S+20.3\ldots )(S^2-1.2\ldots S+1.2\ldots )
(S-7.9\ldots )(S-y_1)}{32(S^2-7.6\ldots S+17.4\ldots )(S-y_0)}>0~.$$
Moreover, $s(\ell _4)<0<s(\ell _3)<s(\ell _6)$. 
Hence for 
$S\in (y_0,y_1)$ the three conditions $w_3>0$, $w_4<0$ and $w_6<0$ cannot hold 
true simultaneously. 

In order to prove the lemma for $S\in [y_1,p_1)$ we consider the conditions 

$$w_1>0~~{\rm ,~~i.e.}~~4(S-1)N-MS<0~~~{\rm and}~~~w_3>0~~{\rm ,~~i.e.}
~~-R(S)N+Q(S)M+S-S^2>0~.$$
The point $\Gamma :=\ell _1\cap \ell _3$ has coordinates 
$(M_{\Gamma},N_{\Gamma})$ which equal 
$$(4(S-1)^2S/(5S^3-16S^2+16S-4)~,
~S^2(S-1)/(5S^3-16S^2+16S-4))~.$$
Both coordinates are positive for $S\in [y_1,p_1)$ (the only real zero of 
the denominator equals $0.3\ldots$). The point $\Gamma$ lies above the 
straight line $\ell _4$. Indeed, substituting $(M_{\Gamma},N_{\Gamma})$ for $(M,N)$ 
in the left-hand side of the equation of $\ell_4:P(S)N-R(S)M+Q(S)=0$ with 
$P(S)>0$ yields 

$$\frac{3(S^2-2.5\ldots S+3.8\ldots )(S^2-0.5\ldots S+0.2\ldots )
(S-1.2\ldots )}{4(S^2-2.8\ldots S+2.1\ldots )(S-0.3\ldots )}>0~.$$
One has $s(\ell _4)<0<s(\ell _3)<s(\ell _1)$; 
the last inequality  follows from 

$$\frac{S}{4(S-1)}-\frac{Q(S)}{R(S)}=
-\frac{5(S^2-2.8\ldots S+2.1\ldots )(S-0.3\ldots )}
{4(S-5.2\ldots )(S-1)(S-0.7\ldots )}>0~.$$
Hence for 
$S\in [y_1,p_1)$ the sector $\{ w_1>0\} \cap \{ w_3>0\}$ does not intersect  
the half-plane $\{ w_4<0\} =\ell _4^-$, i.e. the three 
conditions $w_1>0$, $w_3>0$ and $w_4<0$ do not hold simultaneously. 
\end{proof}


\section{Appendix. Proofs of 
Lemmas~\protect\ref{basiclemma} 
and \protect\ref{nextlemma}} 

\begin{proof}[Proof of Lemma~\protect\ref{basiclemma}]
Denote by $0<x_1<x_2<x_3<x_4$ the real roots of $P$. 
We are looking first for a polynomial 
$U^0(x)$ of the form $(P(x)+ax^8-bx^k+c)/(1+a)$ 
having a quadruple root $x_0>0$, 
where $k=1$ in Cases 3, 5 and 6, $k=3$ in Case~2, $k=5$ in Case~4, and 
$a>0$, $b>0$, $c>0$. The signs of $a$, $b$ and $c$ imply that 
$U^0$ defines the same sign 
pattern as $P$. The polynomial $U$ is obtained from $U^0$ by suitable rescaling 
and multiplication by a positive constant 
which does not change the sign pattern. 

For $x=x_0$ the polynomial $U^0$ satisfies the conditions 
$(U^0)'=(U^0)''=(U^0)'''=0$ 
which read:

\begin{equation}\label{EQQ}
\begin{array}{lll} 
~~k=1&&\\ 
P'(x)+8ax^7-b=0&P''+56ax^6=0&P'''+336ax^5=0\\ 
~~k=3&&\\ 
P'(x)+8ax^7-3bx^2=0&P''+56ax^6-6bx=0&P'''+336ax^5-6b=0\\ 
~~k=5&&\\ 
P'(x)+8ax^7-5bx^4=0&P''+56ax^6-20bx^3=0&P'''+336ax^5-60bx^2=0
\end{array}
\end{equation}
Consider first Cases 5 and 6, hence $k=1$. One eliminates $a$ 
from the last two equations which gives
$xP'''(x)=6P''(x)$. The polynomial $P'$ has exactly three positive roots 
$\mu _1<\mu _2<\mu _3$, $\mu _j\in (x_j,x_{j+1})$. 
Indeed, by Rolle's theorem it has at least three and by  
Descartes' rule of signs it has at most three of them. So for $x>x_4$ 
(resp. $x>\mu _3$) the polynomial $P$ (resp. $P'$) is positive. 

The polynomial $P''$ has at least two real roots 
$\xi _1<\xi _2$, $\xi _j\in (\mu _j,\mu _{j+1})$ (again by Rolle's theorem). 
By Descartes' rule of signs the polynomial $P''$ 
has at most three positive roots. The sign of the coefficient of $x^2$ in $P$ 
is negative, therefore 
$P''$ has exactly three positive roots. The third of them $\xi _3$  
is in $(0,\xi _1)$. Indeed, to the right of $\xi _2$ the number of 
positive roots of $P''$ must be even because for $x>0$ sufficiently large 
$P$ is convex. So $0<\xi _3<\xi _1<\xi _2$.   
 
The polynomial $P'''$ has real roots 
$\zeta _1 \in (\xi _3,\xi _1)$ and $\zeta _2\in (\xi _1,\xi _2)$. By  
Descartes' rule of signs it has at most three positive roots in Case~6 
and at most two in Case~5. In Case~6, as $P'''$ must have an even number of 
roots to the right of $\xi _2$ ($P'$ is convex for $x>0$ sufficiently large), 
the three positive roots $\zeta _3<\zeta _1<\zeta _2$ of $P'''$ belong 
respectively to the intervals 
$(0,\xi _3)$, $(\xi _3,\xi _1)$ and $(\xi _1,\xi _2)$. 

Hence the signs of $P'''(\xi _1)$ and $P'''(\xi _2)$ are opposite and 
$xP'''-6P''$ changes sign at some 
point $x_0\in (\xi _1,\xi _2)$. 

In Case~3 one has again $k=1$. The sign patterns $\sigma _3$ and 
$\sigma _5$ differ only in their third position. The proof resembles the one 
in Cases~5 and 6 yet Descartes' rule of signs allows more positive roots 
for $P'$, $P''$ and $P'''$.  

Denote by $p(P')$ the number of positive roots of $P'$. 
Combining Rolle's theorem and Descartes' rule of signs 
one understands that it is possible to encounter 
only one of the following triples $(p(P'),p(P''),p(P'''))$:

$$i)~(3,5,4)\hspace{1cm}ii)~(3,3,4)\hspace{1cm}iii)~(3,3,2)\hspace{1cm}
iv)~(5,5,4)~.$$
In case {\em iii)} the proof is carried out in exactly the same way as for 
Case~5. In the other cases one performs analogous reasoning with only 
difference the two more positive roots of $P''$ and $P'''$ in case {\em i)}, 
of $P'''$ in case {\em ii)} or of $P'$, $P''$ and $P'''$ in case {\em iv)}. 
For parity reasons the two more roots of the corresponding derivative $P^{(j)}$ 
(compared to their number in the proof of Case~5) must belong to one and the 
same interval of $[0,\infty )$ defined by $0$, $\infty$ and the positive 
roots of $P^{(j-1)}$. One proves as for Case~5 that the signs of $P'''$ 
at two consecutive roots of $P''$ are opposite, hence $xP'''-6P''$ 
changes sign at some 
point $x_0$ from the interval between these two roots.   

Consider Case~2, hence $k=3$. Eliminating $b$ from equations (\ref{EQQ}) yields:

$$2P'-xP''=40ax^7~~{\rm and}~~P''-xP'''=280ax^6~.$$
Eliminating $a$ from the last two equations gives the equation

$$14P'-8xP''+x^2P'''=(14P'-2xP'')-(x/2)(14P'-2xP'')'=0~.$$
The polynomial $P'$ has at most four 
positive roots (by Descartes' rule of signs), 
and at least three of them (denoted by $\mu _j$) belong to   
the intervals $(x_j,x_{j+1})$, 
$j=1$, $2$ and $3$, hence the fourth one $\mu _0$ is in $(0,x_1)$ 
(because $P'(0)>0$). 
The polynomial 
$P''$ has positive roots $\xi _{\nu}\in (\mu _{\nu},\mu _{\nu +1})$, 
$\nu =1$, $2$, and $\xi _0\in (\mu _0,\mu _1)$. Hence the 
polynomial $S:=14P'-2xP''$ has different signs at $\mu _{\nu}$ and 
$\mu _{\nu +1}$ for $\nu =1$ and $2$, hence it has roots 
$\delta _{\nu}\in (\mu _{\nu},\mu _{\nu +1})$, its derivative has opposite 
signs at $\delta _1$ and $\delta _2$, so $S-(x/2)S':=14P'-8xP''+x^2P'''$ 
has a real root $x_0\in (\mu _1,\mu _3)$.  

Consider Case~4, hence $k=5$.  
One first eliminates $b$ (see equations (\ref{EQQ})):

$$4P'-xP''=24ax^7~~\, \, {\rm and}~~\, \, 3P''-xP'''=168ax^6~.$$
Eliminating after this $a$ results in 

$$28P'-10xP''+x^2P'''=(28P'-4xP'')-(x/4)(28P'-4xP'')'=0~.$$
Similarly to the proof in Case~2 one shows that the polynomial 
$28P'-10xP''+x^2P'''$ 
has a positive root $x_0$.

After the number $x_0$ is found, one finds first $a$ and then $b$ 
from system (\ref{EQQ}). Now we have to justify the positive signs of $a$ and 
$b$ (and after this the one of $c$ as well). To this end 
we set $a=ta_*$, $b=tb_*$, where $t>0$, and 
we consider the family 
of polynomials $R_t(x):=P(x)+t\psi _k(x)$ with 
$\psi _k:=a_*x^8-b_*x^k$, $k=1$, $3$ or $5$. We suppose that for some $t>0$ 
the polynomial $R_t$ has a triple critical point at $x_0$. 
Hence for a suitably chosen 
$c$ the polynomial $R_t+c$ has a quadruple root at $x_0$. 
 
Consider the function $\psi _k$ for $x>0$. For 
$a_*\geq 0$, $b_*\leq 0$ and $a_*-b_*>0$ it is increasing and convex, 
for $a_*\leq 0$, $b_*\geq 0$ and $a_*-b_*<0$ it is  
decreasing and concave (for $a_*=0$ and $k=1$ it is linear, i.e. convex and 
concave at the same time). For $a_*>0$ and $b_*>0$ (resp. for $a_*<0$ and 
$b_*<0$) 
it has a minimum (resp. a maximum) at 
$\lambda _k:=(kb_*/8a_*)^{1/(8-k)}$ with $\psi _k(x)<0$ for $x\in (0,\lambda _k]$ 
(resp. 
with $\psi _k(x)>0$ for $x\in (0,\lambda _k]$). 

Consider the family of polynomials $R_t$, where $t$ is 
supposed to belong to an interval $[0,\alpha )$ such that the sign pattern 
defined by the coefficients of $R_t$ is the one of $P$. We keep the 
same notation for the positive roots of $R_t$ and its derivatives 
as the one for $P$. Then:

A) If $\psi _k$ is decreasing on $[\mu _2,\mu _3]$, then as $t$ increases, 
$\mu _2$ moves to the left and $\mu _3$ to the right;

B)  If $\psi _k$ is increasing on $[\mu _1,\mu _2]$, then as $t$ increases, 
$\mu _1$ moves to the left and $\mu _2$ to the right.

In both cases A) and B) it is impossible to have the three positive 
roots of $R_t'$ coalescing into a single critical point of $R_t$. 
If $a_*\geq 0$, $b_*\leq 0$ and $a_*-b_*>0$, then case B) takes place. 
If $a_*\leq 0$, $b_*\geq 0$ and $a_*-b_*<0$, then case A) takes place. 
If $a_*<0$ and $b_*<0$, then at least one of cases A) or B) takes place. 
Hence only for $a_*>0$ and $b_*>0$ can one have a critical point of $R_t$ 
of multiplicity $3$. This implies that $a>0$ and $b>0$. 
Besides, $\lambda _k\in (\mu _1,\mu _3)$. Hence 
$R_t(\mu _1)<0$ (because $P(\mu _1)<0$ and $\psi _k(\mu _1)<0$)  
and to have $U^0(x_0)=0$ one has to choose $c>0$. 
\end{proof} 

\begin{proof}[Proof of Lemma~\protect\ref{nextlemma}]
Prove part (1). Consider the one-parameter family 
of polynomials $U_t:=U-t(x-1)^4$, $t\geq 0$. The first four 
coefficients do not depend on $t$ (they are the same as the ones of $U$). 
The signs of the five coefficients of $-(x-1)^4$ are $(-,+,-,+,-)$. 
Hence the first $8$ components of the sign pattern of $U_t$ 
do not depend on $t$ 
and in the family $U_t$ for some $t>0$, due to the decreasing of the value 
of $U_t$ as $t$ increases, one of the two things 
takes place first: 

a) one has $U_t(0)=0$ or 

b) $U_t$ has one or two negative roots, 
each of them of even multiplicity. 

One can notice that 
the family $U_t$ contains no polynomial 
with six positive roots (counted with multiplicity) because there are 
four or five sign changes in the sign pattern of $U_t$ 
(the sign pattern of $U_t$ is obtained 
from $\sigma _2^r$, $\sigma _4$ or $\sigma _6^r$ by replacing the last 
component by $+$, $0$ or $-$).  

If a) takes place for $t_0>0$, then as $U_t'(0)>0$, the root of $U_t$ at $0$ is 
simple and $U_t$ has one or several negative roots 
whose total multiplicity is odd. 
Hence 
for some $t_1\in (0,t_0)$, b) has taken place. Therefore in the family 
$U_t$ there exists (for some $t>0$) a polynomial of the form (\ref{polyW})
which realizes the pattern $\sigma _2^r$, $\sigma _4$ or $\sigma _6^r$. 

Prove part (2) of the lemma. 
Suppose that the polynomial $U$ realizes the sign pattern 
$\sigma _3^r=(+,-,-,+,-,+,-,+,+)$. Consider the family 
$U^*_t=U+tx(x-1)^4$, $t>0$. The signs of the coefficients of 
$x(x-1)^4$ are $(0,0,0,+,-,+,-,+,0)$, so the sign pattern of $U^*_t$ is 
$\sigma _3^r$ for any $t>0$. The value of $U^*_t$ increases 
(linearly with $t$) for each 
$x>0$, $x\neq 1$ fixed, and decreases for each $x<0$ fixed. 
Hence for some $t>0$ 
the polynomial $U^*_t$ has one or two negative roots each 
of even multiplicity. For this value of $t$ the polynomial $U^*_t$ has the 
form (\ref{polyW}).

The proof of part (3) resembles the one of part (2). 
Suppose that the polynomial $U$ realizes the sign pattern 
$\sigma _6^r=(+,-,-,+,-,+,+,+,+)$. The difference between $\sigma _6^r$ and 
$\sigma _3^r$ is in the sign of the coefficient of $x^2$. Hence in the family 
$U^*_t$ there is a polynomial with a quadruple root at $1$, with one or two 
negative roots of even multiplicity and with coefficients defining 
either one of the sign patterns $\sigma _3^r$, $\sigma _6^r$ or the 
sign pattern $\sigma ^*$ (the sign of the coefficient of $x^2$ in $U^*_t$ 
might change for some value of $t$). In all three cases this is a polynomial 
of the form (\ref{polyW}).
\end{proof}

\end{document}